\documentclass[a4paper, leqno, 10pt]{amsart}

\usepackage[T1]{fontenc}	
\usepackage{graphicx}
\usepackage{amssymb,centernot}
\usepackage{upgreek}
\usepackage{mathrsfs}
\usepackage[usenames,dvipsnames]{xcolor}
\usepackage[inline,shortlabels]{enumitem}
\usepackage[sort,numbers]{natbib}						
\usepackage{verbatim}								
\usepackage{dsfont}									

\usepackage{pdflscape}
\usepackage{afterpage}

\usepackage{anyfontsize}
\definecolor{DarkGray}{gray}{0.35}
\usepackage[
						colorlinks,				
						breaklinks,unicode,
						hypertexnames=false,
						citecolor=OliveGreen,
						linkcolor=Maroon
						]{hyperref} 

\usepackage{multirow}

\usepackage{xypic}
\usepackage{mathtools}
\usepackage{xspace}

\usepackage{booktabs}

\usepackage{tikz}

\usetikzlibrary{calc, positioning, shapes.geometric, patterns, cd}

\newlength{\hatchspread}
\newlength{\hatchthickness}
\newlength{\hatchshift}
\newcommand{\hatchcolor}{}
\tikzset{hatchspread/.code={\setlength{\hatchspread}{#1}},
         hatchthickness/.code={\setlength{\hatchthickness}{#1}},
         hatchshift/.code={\setlength{\hatchshift}{#1}},
         hatchcolor/.code={\renewcommand{\hatchcolor}{#1}}}
\tikzset{hatchspread=3pt,
         hatchthickness=0.4pt,
         hatchshift=0pt,
         hatchcolor=black}
\pgfdeclarepatternformonly[\hatchspread,\hatchthickness,\hatchshift,\hatchcolor]
   {custom north west lines}
   {\pgfqpoint{\dimexpr-2\hatchthickness}{\dimexpr-2\hatchthickness}}
   {\pgfqpoint{\dimexpr\hatchspread+2\hatchthickness}{\dimexpr\hatchspread+2\hatchthickness}}
   {\pgfqpoint{\dimexpr\hatchspread}{\dimexpr\hatchspread}}
   {
    \pgfsetlinewidth{\hatchthickness}
    \pgfpathmoveto{\pgfqpoint{0pt}{\dimexpr\hatchspread+\hatchshift}}
    \pgfpathlineto{\pgfqpoint{\dimexpr\hatchspread+0.15pt+\hatchshift}{-0.15pt}}
    \ifdim \hatchshift > 0pt
      \pgfpathmoveto{\pgfqpoint{0pt}{\hatchshift}}
      \pgfpathlineto{\pgfqpoint{\dimexpr0.15pt+\hatchshift}{-0.15pt}}
    \fi
    \pgfsetstrokecolor{\hatchcolor}
    \pgfusepath{stroke}
   }

\pgfdeclarepatternformonly[\hatchspread,\hatchthickness,\hatchshift,\hatchcolor]
   {custom north east lines}
   {\pgfqpoint{\dimexpr-2\hatchthickness}{\dimexpr-2\hatchthickness}}
   {\pgfqpoint{\dimexpr\hatchspread+2\hatchthickness}{\dimexpr\hatchspread+2\hatchthickness}}
   {\pgfqpoint{\dimexpr\hatchspread}{\dimexpr\hatchspread}}
   {
    \pgfsetlinewidth{\hatchthickness}
    \pgfpathmoveto{\pgfqpoint{\dimexpr\hatchshift-0.15pt}{-0.15pt}}
    \pgfpathlineto{\pgfqpoint{\dimexpr\hatchspread+0.15pt}{\dimexpr\hatchspread-\hatchshift+0.15pt}}
    \ifdim \hatchshift > 0pt
      \pgfpathmoveto{\pgfqpoint{-0.15pt}{\dimexpr\hatchspread-\hatchshift-0.15pt}}
      \pgfpathlineto{\pgfqpoint{\dimexpr\hatchshift+0.15pt}{\dimexpr\hatchspread+0.15pt}}
    \fi
    \pgfsetstrokecolor{\hatchcolor}
    \pgfusepath{stroke}
   }

\makeatletter
\newcommand*{\centerfloat}{%
  \parindent \z@
  \leftskip \z@ \@plus 1fil \@minus \textwidth
  \rightskip\leftskip
  \parfillskip \z@skip}
\makeatother

\usepackage{xparse}

\ExplSyntaxOn
\NewDocumentCommand{\makeabbrev}{mmm}
 {
  \yoruk_makeabbrev:nnn { #1 } { #2 } { #3 }
 }

\cs_new_protected:Npn \yoruk_makeabbrev:nnn #1 #2 #3
 {
  \clist_map_inline:nn { #3 }
   {
    \cs_new_protected:cpn { #2 } { #1 { ##1 } }
   }
 }
 \ExplSyntaxOff

\makeabbrev{\textbf}{tbf#1}{a,b,c,d,e,f,g,h,i,j,k,l,m,n,o,p,q,r,s,t,u,v,w,x,y,z,A,B,C,D,E,F,G,H,I,J,K,L,M,N,O,P,Q,R,S,T,U,V,W,X,Y,Z}

\makeabbrev{\textbf}{bf#1}{a,b,c,d,e,f,g,h,i,j,k,l,m,n,o,p,q,r,s,t,u,v,w,x,y,z,A,B,C,D,E,F,G,H,I,J,K,L,M,N,O,P,Q,R,S,T,U,V,W,X,Y,Z}

\makeabbrev{\textsf}{tsf#1}{a,b,c,d,e,f,g,h,i,j,k,l,m,n,o,p,q,r,s,t,u,v,w,x,y,z,A,B,C,D,E,F,G,H,I,J,K,L,M,N,O,P,Q,R,S,T,U,V,W,X,Y,Z}

\makeabbrev{\mathsf}{mss#1}{a,b,c,d,e,f,g,h,i,j,k,l,m,n,o,p,q,r,s,t,u,v,w,x,y,z,A,B,C,D,E,F,G,H,I,J,K,L,M,N,O,P,Q,R,S,T,U,V,W,X,Y,Z}

\makeabbrev{\mathfrak}{mf#1}{a,b,c,d,e,f,g,h,i,j,k,l,m,n,o,p,q,r,s,t,u,v,w,x,y,z,A,B,C,D,E,F,G,H,I,J,K,L,M,N,O,P,Q,R,S,T,U,V,W,X,Y,Z}

\makeabbrev{\mathrm}{mrm#1}{a,b,c,d,e,f,g,h,i,j,k,l,m,n,o,p,q,r,s,t,u,v,w,x,y,z,A,B,C,D,E,F,G,H,I,J,K,L,M,N,O,P,Q,R,S,T,U,V,W,X,Y,Z}

\makeabbrev{\mathbf}{mbf#1}{a,b,c,d,e,f,g,h,i,j,k,l,m,n,o,p,q,r,s,t,u,v,w,x,y,z,A,B,C,D,E,F,G,H,I,J,K,L,M,N,O,P,Q,R,S,T,U,V,W,X,Y,Z}

\makeabbrev{\mathcal}{mc#1}{A,B,C,D,E,F,G,H,I,J,K,L,M,N,O,P,Q,R,S,T,U,V,W,X,Y,Z}

\makeabbrev{\mathbb}{mbb#1}{A,B,C,D,E,F,G,H,I,J,K,L,M,N,O,P,Q,R,S,T,U,V,W,X,Y,Z}

\makeabbrev{\mathscr}{ms#1}{A,B,C,D,E,F,G,H,I,J,K,L,M,N,O,P,Q,R,S,T,U,V,W,X,Y,Z}

\makeabbrev{\mathrm}{#1}{
Id,id,ran,rk,diag,stab,ann,conv,pr,ev,tr,End,Hom,sgn,im,op,can,fin,ext,red,tot,
rot,usc,lsc,Lip,LocLip,lip,bSymLip,osc,AC,loc,uloc,spec,coz,z,ul,
supp,Opt,Adm,Cpl,Geo,GeoSel,GeoOpt,GeoAdm,GeoCpl,reg,
bd,co,Ric,Exp,dExp,dist,seg,Seg,cut,fcut,Cut,SDiff,Iso,Isom,diam,cl,Homeo,Diff,Der,vol,dvol,inj,relint, Graph, sub,codim,
var,law,Poi,Gam,pa,so,iso,fs,inv,pqi,mix,
TestF,
}

\makeabbrev{\mathsf}{#1}{DP,CD,BE,MCP,Ent,wMTW,MTW,RCD,ncRCD,QCD,EVI,Irr,IH,SC,wFe,VA,UP,Curv,Alex,CAT, Var}

\let\epsilon\varepsilon

\let\temp\phi
\let\phi\varphi
\let\varphi\temp

\newcommand{\diff}{\mathop{}\!\mathrm{d}}

\DeclareSymbolFont{symbolsC}{U}{pxsyc}{m}{n}
\SetSymbolFont{symbolsC}{bold}{U}{pxsyc}{bx}{n}
\DeclareFontSubstitution{U}{pxsyc}{m}{n}
\DeclareMathSymbol{\medcirc}{\mathbin}{symbolsC}{7}
\DeclareSymbolFont{symbolsZ}{OMS}{pxsy}{m}{n}
\SetSymbolFont{symbolsZ}{bold}{OMS}{pxsy}{bx}{n}
\DeclareFontSubstitution{OMS}{pxsy}{m}{n}

\newcommand{\N}{{\mathbb N}}
\newcommand{\R}{{\mathbb R}}

\allowdisplaybreaks

\usetikzlibrary{shapes.misc}
\tikzset{cross/.style={cross out, draw=black, minimum size=2*(#1-\pgflinewidth), inner sep=0pt, outer sep=0pt},
cross/.default={4pt}}

\newcommand{\comma}{\,\,\mathrm{,}\;\,}

\newcommand{\fstop}{\,\,\mathrm{.}}

\newcommand{\cdc}{\Gamma}

\usepackage{scrextend}						

\newcommand{\Cyl}[1]{\mcF\mcC^\infty_b(#1)}

\newcommand{\QP}{{\mu}}

\newcommand{\dUpsilon}{{\boldsymbol\Upsilon}}

\newcommand{\U}{\dUpsilon}
\newcommand{\sine}{\mathsf{sine}}

\newcommand{\E}{\mathcal E}
\newcommand{\F}{\mathcal F}

 \newcommand{\cyl}{\Cyl{\mathcal D}}

\theoremstyle{plain}	
\newtheorem{thm}{Theorem}[section]
\newtheorem*{thm*}{Theorem}
\newtheorem*{mthm*}{Main Theorem}

\newtheorem*{cor*}{Corollary}

\theoremstyle{definition}
\newtheorem*{defs*}{Definition}

\theoremstyle{remark}

\newtheorem{rem}[thm]{\bf Remark}

\newtheorem*{asm*}{{\bf List of Assumptions}}

\renewcommand{\paragraph}[1]{\medskip\emph{#1}.}

\newcommand{\X}{\R}
\newcommand{\cquad}{\comma\quad}

\makeatletter
\@namedef{subjclassname@2020}{%
  \textup{2020} Mathematics Subject Classification}
\makeatother

\begin{document}
\title[Spectral Gap of Dyson Brownian Motion]{The Infinite Dyson Brownian Motion with $\beta=2$ Does Not Have a Spectral Gap}

\author[K.~Suzuki]{Kohei Suzuki}
\thanks{\hspace{-5.5mm}  Department of Mathematical Science, Durham University, South Road, DH1 3LE, United Kingdom
\\
\hspace{2.0mm} E-mail: kohei.suzuki@durham.ac.uk
}

\keywords{\vspace{2mm}infinite Dyson Brownian motion, spectral gap}


\begin{abstract} 
We prove that the Dirichlet forms associated with the
unlabelled infinite Dyson Brownian motion with the
inverse temperature $\beta=2$ do not have a spectral gap.
\end{abstract}

\maketitle

\vspace{-8mm}
\section*{}
The interacting particle system~$\mathbb X_t=(X_t^i)_{i \in \N}$ discussed in this article is {\it formally} described as the following  stochastic differential equation of infinitely many particles in $\R$:
\begin{align}  \label{d:DBMS} \tag{DBM}
\diff X_t^i=     \frac{\beta}{2}\sum_{j: j \neq i} \frac{1}{X_t^i- X_t^j} \diff t + \diff B^i_t, \quad i \in \N \comma 
\end{align}
 where $(B_t^i: i \in \N)$ is the family of infinitely many independent Brownian motions in~$\R$ and $\beta>0$ is a positive constant called inverse temperature. The solution~$\mathbb X$ to \eqref{d:DBMS} is called {\it infinite Dyson Brownian motion with inverse temperature $\beta$}, named after Dyson~\cite{Dys62}, which has a particular importance in relation to random matrix theory.  It can be thought of as a diffusion process in~the space $\U$ of locally finite point measures (called {\it configuration space}) by dropping the labelling via the map $(x_i)_{i=1}^\infty \mapsto \sum_{i=1}^\infty \delta_{x_i}$, which is called {\it unlabelled solution} and denoted by $\mathsf X$.  
Over these thirty five years, the construction of weak/strong solutions and their uniqueness have been studied e.g., in \cite{Spo87, NagFor98, KatTan10, Osa96, Osa12, Osa13, Tsa16}. 

A question that is addressed in this article is a spectral gap of the unlabelled solution~$\mathsf X$ in the case $\beta=2$.
In this case,  a weaker property, what is called {\it irreducibility} (called also {\it ergodicity}, or {\it convergence to equilibrium}),  has been recently settled affirmatively in \cite{OsaTsu21, OsaOsa23, Suz23}. In particular,  the law~of the time marginal~$\mathsf X_t$ converges to an equilibrium measure as $t \to \infty$, which is the law of the $\sine_2$ point process.
Any aspect of quantitative rate of this convergence, however, remains uncharted so far. 
The objective of this article is to provide a negative result for spectral gaps of the unlabelled solution~$\mathsf X$ with $\beta=2$. 
\vspace{-2mm}
\subsection*{List of Notation}
\begin{itemize}
\item $\mathcal D=\mathcal C_c^\infty(\X)$ for the space of real-valued compactly supported smooth functions in $\X$;
\item $\mathcal C_b^\infty(\R^k)$ for the space of real-valued smooth functions in $\R^k$ $(k \in \N)$ whose arbitrary order of derivative is continuous and bounded;
\item $\cdc^{\X}(u, v)$ for the square field of functions $u, v:\X \to \R$ defined as
$\cdc^\R(u, v):= (\frac{\diff}{\diff x}u)(\frac{\diff}{\diff x}v)$;
\end{itemize}
\begin{itemize}
\item $\U$ for the {\it configuration space over $\X$}, i.e., the space of Radon point measures  on $\X$;
\item $\QP$ for the law of the {\it $\sine_2$ point process}. Namely, $\QP$ is the law of the determinantal point process whose $n$-point correlation function is 
$$\rho^{(n)}(x_1, \ldots, x_n) = \mathsf{det}\Bigl(K(x_i, x_j)_{i, j =1}^n\Bigr) \cquad K(x_i, x_j)=\frac{\sin \pi(x_i-x_j)}{\pi(x_i-x_j)} \ ;$$
\item $L^p(\U, \QP)$ $(1 \le p <\infty)$ for the $\QP$-equivalence classes of functions $u:\U \to \R$ with $\|u\|^p_{L^p(\QP)}:=\int_{\U}|u|^p \diff \QP<\infty$;
\item $\Var(u)$ for the {\it variance} of $u \in L^2(\U, \QP)$ defined as
$\Var(u):=\int_{\U} u^2 \diff \QP - \Bigl(\int_{\U} u \diff \QP \Bigr)^2$;
\item $u^*:\U \to \R$ for {\it linear statistics of $u: \X \to \R$} defined as 
$$u^*(\gamma):=\int_{\X} u(x) \diff \gamma(x) \cquad \gamma \in \U \ ;$$
\item $\Cyl{\mathcal D}$ for the space of {\it cylinder functions} $U: \U \to \R$:
\begin{align*}
U=\Phi(u_1^*, \ldots, u_k^*) \comma \  \{u_1,\ldots, u_k\} \subset \mathcal D \comma \  \Phi \in \mathcal C_b^\infty(\R^k) \comma \  k \in \N \ ;
\end{align*}
\item $\cdc^{\U}$ for the {\it square field operator in $\U$}: 
\begin{align*}
\cdc^{\U}(U):=\sum_{i, j=1}^k \partial_{i}\Phi(u_1^*, \ldots, u_k^*) \partial_{j}\Phi(u_1^*, \ldots, u_k^*) \cdc^\R(u_i, u_j)^* \cquad U \in \Cyl{\mathcal D} \ ;
\end{align*}
\item $\E$ for the functional $\E:  \Cyl{\mathcal D}\times \Cyl{\mathcal D}  \to \R$ defined as 
\begin{align} \label{e:EF}
\E(u,v):=\frac{1}{2}\int_{\U}\cdc^{\U}(u, v) \diff \QP \cquad u, v \in \Cyl{\mathcal D} \comma
\end{align}
where $\cdc^{\U}(u, v) :=\frac{1}{4}\bigl(\cdc^\U(u+v)-\cdc^\U(u-v)\bigr)$. 
We write~$\E(u, u)=\E(u)$.
\end{itemize}
Let $\mathcal Q: L^2(\U, \QP) \times L^2(\U, \QP) \to \R \cup\{\infty\}$  be a symmetric bilinear function with a dense domain~$\mathcal F:=\{u\in L^2(\U, \QP): \mathcal Q(u, u)<\infty\}$ and $\mathcal Q(u, u) \ge 0$ for every $u \in \mathcal F$. It is called {\it closed} if the space~$\mathcal F$ endowed with the norm~$\|\cdot\|_\mathcal F$ defined as $\|\cdot\|^2_\mathcal F:=\mathcal Q(\cdot) + \|\cdot\|_{L^2(\QP)}^2$ is a real~Hilbert space. A pair~$(\mathcal Q, \mathcal F)$ is a {\it closed extension of $(\E, \Cyl{\mathcal D})$} if $(\mathcal Q, \mathcal F)$ is closed and 
$$ \cyl \subset \mathcal F \cquad \mathcal Q= \E \quad \text{on} \quad \cyl\times \cyl \fstop$$ 
In the rest of the paper, we keep the same symbol~$\E$ for extensions and simply say that $(\E, \mathcal F)$ is a closed extension of $(\E, \cyl)$. 
The unlabelled solution $\mathsf X$ of \eqref{d:DBMS} with $\beta=2$ starting at a particular class of admissible initial conditions has been identified with the diffusion process properly associated with a closed extension of~$(\E, \Cyl{\mathcal D})$, see~\cite[Thm.~24]{Osa12}. 
\smallskip

\begin{thm*}\label{t:maind} Any closed extension~$(\E, \mathcal F)$ of $(\E, \Cyl{\mathcal D})$ does not have a spectral gap, i.e., 
$$\inf_{\substack{u \in \mathcal F \\ u \neq 0}} \frac{\E(u)}{\Var(u)} = 0 \fstop$$
\end{thm*}

\section{Proof of the Theorem~\ref{t:maind}}

In the rest of the arguments, we often use the following facts about the $\sine_2$ point process. 
\begin{itemize}
\item {\bf (intensity measure)} If $u \in L^1(\R)$, then $u^* \in L^1(\U, \QP)$ and 
\begin{align}\label{d:I}
\int_\U u^* \diff \QP = \int_\X u \diff x  \fstop 
\end{align}
\item {\bf (two-point correlation)} for every Borel measurable $u: \X \to \R$, 
\begin{align}\label{d:T}
\int_\U \sum_{\substack{x, y \in \gamma \\ x \neq y}} u(x)u(y) \diff \QP(\gamma) = \int_{\X^2} u(x)u(y) \rho^{(2)}(x, y)\diff x \diff y \quad (\le \infty)\comma
\end{align}
 where $\rho^{(2)}$ is the two-point correlation function given by
$$\rho^{(2)}(x, y) = 1- \frac{\sin^2\bigl(\pi(x-y)\bigr)}{\pi^2(x-y)^2} \fstop$$
\end{itemize}
 \begin{proof}[Proof of the Theorem~\ref{t:maind}]
Take $u_\sigma(x):=  xe^{-\frac{x^2}{2\sigma^2}}$ and define the associated linear statistics~$U_\sigma: \U \to \R$  as $U_\sigma(\gamma):=u_\sigma^*(\gamma)=\int_{\X} u_\sigma(x) \diff \gamma$. In the following argument, we fix $\sigma$ and simply write $u$ and $U$.

We first compute the variance of $U$. By the intensity formula~\eqref{d:I} and the mean zero property~$\int_\R u \diff x=0$,
\begin{align*}
\Var(U)=\int_{\U}U^2 \diff \QP - \biggl(\int_{\U}U \diff \QP\biggr)^2=\int_{\U}U^2 \diff \QP - \biggl(\int_{\X}u \diff x\biggr)^2 = \int_{\U}U^2 \diff \QP\fstop
\end{align*}
The right-hand side can be further deduced to 
\begin{align} \label{e:VA}
\int_{\U}\biggl(\sum_{x \in \gamma}u(x)\biggr)^2 \diff \QP(\gamma)  
&= \int_{\U}\biggl(\sum_{x \in \gamma}u(x)^2 +\sum_{\substack{x, y \in \gamma\\ x \neq y}}u(x)u(y)\biggr)  \diff \QP(\gamma)
\\
& = \int_{\X}u(x)^2\diff x + \int_{\X^{2}} u(x)u(y) \rho^{(2)}(x,y) \diff x \diff y \notag
\\
&= \frac{\sqrt{\pi}\sigma^3}{2} +  \int_{\X^{2}} u(x)u(y) \rho^{(2)}(x,y) \diff x \diff y \fstop \notag
\end{align}
We now give the evaluation on the second term $\int_{\X^{2}} u(x)u(y) \rho^{(2)}(x,y) \diff x \diff y$. Recalling $\rho^{(2)}(x,y)= 1-\frac{\sin^2(\pi(x-y))}{\pi^2(x-y)^2}$ and $\int_{\X^{2}}u(x)u(y) \diff x \diff y=0$, we focus only on the evaluation of
\begin{align}\label{e:coI}
-\int_{\X^{2}}u(x)u(y)\frac{\sin^2(\pi(x-y))}{\pi^2(x-y)^2} \diff x \diff y \fstop
\end{align}
By the change of variables (the rotation by $45$ degree)
\begin{align*}
x=\frac{1}{\sqrt{2}}(u+v) \cquad y=\frac{1}{\sqrt{2}}(v-u) \comma
\end{align*}
the integral~\eqref{e:coI} comes down to 
\begin{align*}
& -\frac{1}{4}\int_{\X^2} (v^2-u^2)e^{-\frac{u^2+v^2}{2\sigma^2}} \frac{\sin^2(\sqrt{2}\pi u)}{\pi^2 u^2} \diff u \diff v
\\
&=\frac{1}{4}\int_{\X^2} e^{-\frac{u^2+v^2}{2\sigma^2}} \frac{\sin^2(\sqrt{2}\pi u)}{\pi^2} \diff u \diff v  - \frac{1}{4}\int_{\X^2}  v^2e^{-\frac{u^2+v^2}{2\sigma^2}} \frac{\sin^2(\sqrt{2}\pi u)}{\pi^2 u^2} \diff u \diff v
\\
&=: {\rm (I)} + {\rm (II)} \fstop
\end{align*}
By using the following formula (a proof will be given later): 
\begin{align} \label{c:s}
\int_{\R}e^{-\frac{u^2}{2\sigma^2}} \frac{\sin^2(\sqrt{2}\pi u)}{\pi^2} \diff u = \frac{\sigma e^{-4\pi^2 \sigma^2}(e^{4\pi^2 \sigma^2}-1)}{\sqrt{2}\pi^{3/2}}  \comma
\end{align}
the first term can be further computed as
\begin{align*}
{\rm (I)}&=\frac{1}{4}\int_{\R}e^{-\frac{v^2}{2\sigma^2}} \diff v \int_{\R}e^{-\frac{u^2}{2\sigma^2}} \frac{\sin^2(\sqrt{2}\pi u)}{\pi^2} \diff u 
=\frac{\sqrt{2\pi}\sigma}{4}  \times \frac{\sigma e^{-4\pi^2 \sigma^2}(e^{4\pi^2 \sigma^2}-1)}{\sqrt{2}\pi^{3/2}}  
\\
&= \frac{\sigma^{2}}{4\pi}(1-e^{-4\pi^2 \sigma^2}) \fstop
\end{align*}

For the second term, we first note that $\int_{\R} \frac{\sin^2(\sqrt{2}\pi u)}{\pi^2 u^2} \diff u = \sqrt{2}$, which can be immediately seen by the formula $\int_{\R} \frac{\sin^2(u)}{u^2}\diff u = \pi$ and the change of variable. Thus, we have 
\begin{align*}
{\rm (II)}&=- \frac{1}{4}\int_{\R}e^{-\frac{u^2}{2\sigma^2}}\frac{\sin^2(\sqrt{2}\pi u)}{\pi^2 u^2} \diff u \int_{\R} v^2 e^{-\frac{v^2}{2\sigma^2}} \diff v 
=- \int_{\R}e^{-\frac{u^2}{2\sigma^2}}\frac{\sin^2(\sqrt{2}\pi u)}{\pi^2 u^2} \diff u \times \frac{\sqrt{2\pi}\sigma^3}{4}
 \\
 &\ge -\int_{\R}\frac{\sin^2(\sqrt{2}\pi u)}{\pi^2 u^2} \diff u \times \frac{\sqrt{2\pi}\sigma^3}{4}
 = -\frac{\sqrt{\pi}\sigma^3}{2}
 \fstop
\end{align*}
By plugging these estimates into \eqref{e:VA},
\begin{align}
\Var(U) \ge \frac{\sqrt{\pi}\sigma^3}{2}  - \frac{\sqrt{\pi}\sigma^3}{2} +  \frac{\sigma^{2}}{4\pi}(1-e^{-4\pi^2 \sigma^2})  =  \frac{\sigma^{2}}{4\pi}(1-e^{-4\pi^2 \sigma^2}) \fstop
\end{align}

We next compute $\E(U)=(1/2)\int_\U \cdc^\U(U) \diff \QP$. By e.g., \cite[Prop.~4.6]{MaRoe00}, $U=U_\sigma \in \mathcal F$ for every $\sigma>0$ and $\cdc^{\U}(U)=\cdc^{\R}(u)^*$ $\QP$-a.e.. Thus, for $\QP$-a.e.~$\gamma$, 
\begin{align*}
\cdc^\U(U)(\gamma) = \sum_{x \in \gamma}\cdc^\X(u)(x) \le 2\sum_{x \in \gamma}e^{-\frac{x^2}{\sigma^2}}(x)+ \frac{2}{\sigma^4}\sum_{x \in \gamma}x^4e^{-\frac{x^2}{\sigma^2}}(x)\fstop
\end{align*}
By the intensity formula~\eqref{d:I}, we have 
\begin{align}
\E(U)=\frac{1}{2}\int_{\U}\cdc^\U(U)\diff \QP & \le  \int_{\U}\bigg(\sum_{x \in \gamma}e^{-\frac{x^2}{\sigma^2}}(x)+ \frac{1}{\sigma^4}\sum_{x \in \gamma}x^4e^{-\frac{x^2}{\sigma^2}}(x) \bigg)\diff \QP(\gamma) \notag
\\
&=  \int_{\X}\Bigl( e^{-\frac{x^2}{\sigma^2}} +  \frac{1}{\sigma^4} x^4e^{-\frac{x^2}{\sigma^2}}\Bigr)  \diff x \notag
\\
&=\sqrt{\pi} \sigma + \frac{3\sqrt{\pi}\sigma}{4}\fstop \notag
\end{align}
Therefore, we conclude 
\begin{align*}
\inf_{\substack{F \in \mathcal F\\F\neq 0}}\frac{\E(F)}{\Var(F)} 
\le \inf_{\sigma >0}\frac{\E(U_\sigma)}{\Var(U_\sigma)} \le  \frac{\sqrt{\pi} \sigma + \frac{3\sqrt{\pi}\sigma}{4}}{ \frac{\sigma^{2}}{4\pi}(1-e^{-4\pi^2 \sigma^2}) }\xrightarrow{\sigma \to \infty} 0 \fstop
\end{align*}
\end{proof}

\begin{proof}[Proof of \eqref{c:s}]
By the formula~$\sin^2 u=(1-\cos (2u))/2$, 
\begin{align*}
\int_{\R} e^{-\frac{u^2}{2\sigma^2}} \frac{\sin^2(\sqrt{2}\pi u)}{\pi^2} \diff u
&=\frac{1}{2\pi^2}\int_{\R} e^{-\frac{u^2}{2\sigma^2}} \bigl(1-\cos(2\sqrt{2}\pi u) \bigr)\diff u
\\
&=\frac{1}{2\pi^2}\sqrt{2\pi}\sigma - \frac{1}{2\pi^2}\int_{\R}e^{-\frac{u^2}{2\sigma^2}} \cos(2\sqrt{2}\pi u) \diff u \fstop
\end{align*}
Let $h(a):=\int_{0}^\infty e^{-b{u^2}} \cos(a u) \diff u$ for $a, b>0$. As the integrant $e^{-b{u^2}} \cos(a u)$ is symmetric in variable $u$, we have $2h(a)=\int_{\R} e^{-b{u^2}} \cos(a u) \diff u$. It is straightforward to check 
$$\partial_a h = -\frac{a}{2b} h \cquad h(0)=\int_0^\infty e^{-bu^2}\diff u = \frac{\sqrt{\pi}}{2\sqrt{b}} \fstop$$
 Solving this first-oder ordinary differential equation, we get $2h(a)=\frac{\sqrt{\pi}}{\sqrt{b}}e^{-\frac{a^2}{4b}}$.  Plugging $a=2\sqrt{2}\pi$  and  $b=1/(2\sigma^2)$, the sought formula is obtained.  
\end{proof}

\section{Concluding remark}

\begin{rem}[growth of variance \& spectral gap]
Let $\QP$ be the law of a  point process in $\R$ with constant intensity, $\Var_\QP$ be the variance with respect to $\QP$, and $\E_\QP$ be the functional defined in~\eqref{e:EF} with $\QP$.   
The proof of the Theorem shows that as long as 
 \begin{align}\label{e:k}\frac{\sigma}{\Var_\QP(u_\sigma^*)} \xrightarrow{\sigma \to \infty}0 \comma
 \end{align}
 any closed extension of $(\E_\QP, \Cyl{\mathcal D})$ (if it exists) does not have a spectral gap. 
 In the case where $\QP$ is the law $\pi$ of the Poisson point process in $\R$ with intensity~$1$, \eqref{e:k} holds true as we have 
$$\Var_{\pi}(u_\sigma^*)=\frac{\sqrt{\pi}\sigma^3}{2} \fstop$$
To find an example $\QP$ with which $(\E_\QP, \Cyl{\mathcal D})$ has a spectral gap, we need to look into point processes having a considerably slower growth of the variance than that of the Poisson point process so that \eqref{e:k} does not hold. This might be related to a property called {\it hyperuniformity}. 
\end{rem}

\section*{Acknowledgement}
The author appreciates Ostap Hryniv for his suggestion about the computation of Equation (5).

\bibliographystyle{alpha}
\bibliography{GeneralMasterBib.bib}

\end{document}